\documentclass[11pt]{article}
\usepackage[a4paper]{anysize}\marginsize{2.5cm}{2.5cm}{1.3cm}{1.2cm}
\pdfpagewidth=\paperwidth \pdfpageheight=\paperheight
\usepackage{pgf,tikz,hyperref}
\usepackage{amsfonts,amssymb,amsthm,amsmath,eucal}
\usetikzlibrary{cd}
\usetikzlibrary{arrows}
\usepackage{caption}

\newcommand\numberthis{\addtocounter{equation}{1}\tag{\theequation}}

\theoremstyle{plain}
\newtheorem{thm}{Theorem}
\newtheorem{theorem}[thm]{Theorem}
\newtheorem*{theoremA}{Theorem A}
\newtheorem*{theoremB}{Theorem B}

\newtheorem{lemma}[thm]{Lemma}
\newtheorem{proposition}[thm]{Proposition}

\theoremstyle{definition}
\newtheorem{definition}[thm]{Definition}
\newtheorem{remark}[thm]{Remark}
\newtheorem{example}[thm]{Example}

\newtheorem{problem}[thm]{Problem}

\newtheorem{property}[thm]{Property}

\newtheorem{thevarthm}[thm]{\varthmname}

\newenvironment{varthm*}[1]{\trivlist\item[]{\bf #1.}\it}{\endtrivlist}

\newcommand\be{\begin{eqnarray*}}
\newcommand\ee{\end{eqnarray*}}

\newcommand\tensor{\otimes}
\newcommand\lra{\longrightarrow}
\newcommand\C{\mathbb C}
\newcommand\K{\mathbb K}

\newcommand\cali{\mathcal I}
\newcommand\cald{\mathcal D}

\renewcommand\P{\mathbb P}

\newcommand\calo{{\mathcal O}}
\newcommand\newop[2]{\def#1{\mathop{\rm #2}\nolimits}}
\newop\edim{edim}
\newop\mult{mult}
\newop\Zeroes{Zeroes}
\newop\Jac{Jac}
\newop\Osc{Osc}
\newcommand\Sh{X_{S}}
\newcommand\To{X_{T}}
\newcommand\Btre{X_{B}}
\newop\Sym{Sym}

\newcommand\eqnref[1]{(\ref{#1})}
\newcommand\wtilde[1]{\widetilde{#1}}

\newcommand\keywords[1]{{\renewcommand\thefootnote{}\footnotetext{\textit{Keywords:} #1.}}}
\newcommand\subclass[1]{{\renewcommand\thefootnote{}\footnotetext{\textit{Mathematics Subject Classification (2010):} #1.}}}

\def\endproof{\hspace*{\fill}\endproofsymbol\endtrivlist}

\def\endproofsymbol{\frame{\rule[0pt]{0pt}{6pt}\rule[0pt]{6pt}{0pt}}}

\begin{document}

\author{Justyna~Szpond}
\title{Unexpected curves and Togliatti-type surfaces}
\date{\today}
\maketitle
\thispagestyle{empty}

\begin{abstract}
   The purpose of this note is to establish a direct link between the theory of unexpected hypersurfaces and varieties
   with defective osculating behavior. We identify unexpected plane curves of degree $4$ as
   sections of a rational surface $\Btre$ of degree $7$ in $\P^5$ with its osculating spaces of order $2$
   which in every point of $\Btre$ have dimension lower than expected. We put this result in perspective
   with earlier examples of surfaces with defective osculating spaces due to Shifrin and Togliatti.
   Our considerations are rendered by an analysis of Lefschetz Properties of ideals associated with the studied surfaces.
\subclass{14C20 \and 14J26 \and 14N05 \and 13A15 \and 13E10 \and 13F20}
\end{abstract}


\section{Introduction}
   In the ground-breaking paper \cite{CHMN} the authors introduced the notion of unexpected curves.
\begin{definition}
We say that a reduced finite set of points $Z\subset \P^2$ \emph{admits an unexpected curve} of degree
$j+1$ if there is an integer $j>0$ such that, for a general point $P$, $jP$ fails to impose the expected
number of conditions on the linear system of curves of degree $j+1$ containing $Z$. That is, $Z$ admits
an unexpected curve of degree $j+1$ if
$$
h^0(\cali_{Z+jP}(j+1))>\max\left\{ h^0(\cali_Z(j+1))-\binom{j+1}{2},0\right\}.
$$
\end{definition}
This notion has been generalized to hypersurfaces in projective spaces of arbitrary dimension
in the sequel paper \cite{HMNT}.

One of the first examples of unexpected curves (in fact this example has motivated this path of research) was discovered
in \cite[Proof of Proposition 7.3]{DIV14}. The recent work \cite{FGSW} shows that it is the unique
example of an unexpected quartic. This example is derived from the $B_3$ root system, see \cite[Example 4.6]{HMNT} and
\cite[section 2]{BMSS}.

On the other hand the study of osculating spaces of projective varieties has a long history, going back to Corrado Segre
and Gaston Darboux. A modern treatment has been given in \cite{Poh62}.
\begin{definition}\label{de: osculating space}
Let $X\subseteq \P^N$ be a smooth, complete, complex variety. Let $m\geq 1$ be an integer. The \emph{$m$-th osculating
space} to $X$ at $P$ is the linear subspace $\Osc^{(m)}_P(X)$ determined at a point $P\in X$ by the partial derivatives
of order $\leq m$ of the coordinate functions, with respect to a system of local parameters for $X$ at $P$, evaluated
at $P$.
\end{definition}
In particular, $\Osc^{(1)}_P(X)$ is just the tangent space of $X$ at $P$.

A somewhat more formal, equivalent description is the following.

Let $X$ be a smooth, complete, complex variety of dimension $n$ and let $L$ be a line bundle on $X$. Let $V=H^0(X,L)$.
The $m$-jet bundle of $L$ is the coherent sheaf
$$
\cali_m(L)=(p_1)_*(p_2^*L\otimes\calo_{X\times X}/\cali_\Delta^{m+1}),
$$
where $\Delta\subseteq X\times X$ is the diagonal and $p_1$, $p_2$ are the projection maps
\begin{center}
\begin{tikzcd}
 & X\times X \arrow[ld, "p_1"'] \arrow{rd}{p_2}&\\
 X & & X
\end{tikzcd}
\end{center}

This sheaf is locally free of rank $\binom{n+m}{n}$ and its fiber at a point $P$  can be identified with
$$
\cali_m(L)_P\simeq H^0(X,L\otimes\calo_X/\frak{m}_P^{m+1}),
$$
where $\frak{m}_P$ is the ideal sheaf of $P\in X$.

Then the $m$-th osculating space at $P$ is the projectivization $\P(j_{k,P}(V))\subseteq \P(V)$ of the image of
the evaluation map
$$
j_{k,P} : V\longrightarrow H^0(X,L\otimes\calo_X/\frak{m}_P^{m+1}).
$$
This agrees with Definition \ref{de: osculating space} in case $X\subset \P^N$ and $L=\calo_X(1)$.

It is natural to expect that the space $\Osc_P^{(m)}(X)$ has at  a general point $P\in X$ the (projective)
dimension $\binom{n+m}{n}-1$. If this dimension is lower than $\binom{n+m}{n}-1$ at every point, then
following  Shifrin \cite{Shi86} we say that $X$ is \emph{hypo-osculating of order $m$}.
   More precisely, the deficiency of osculating spaces of $X$ is measured by the number of Laplace
   equations $X$ satisfies.
\begin{definition}[Laplace equation]
   A projective variety $X\subset \P^N$ of dimension $n$ satisfies $\delta$ independent Laplace equations
   of order $m$, if for a general point $P\in X$ there is
   $$\dim\Osc^{(m)}_PX=\min\left\{\binom{n+m}{n}-1,N\right\}-\delta.$$
\end{definition}

In this note we are interested in surfaces $X\subseteq\P^5$. It is then expected that $\dim \Osc_P^{(2)}(X)=5$
at a general point $P\in X$. It has been noticed, apparently first by Eugenio Togliatti \cite{Tog29} that
there are hypo-osculating surfaces of order $2$ in $\P^5$.

Our main results are the following Theorems.
\begin{theoremA}
Let $Z=\{P_1,\ldots,P_9\}\subseteq \P^2$ be the $B_3$ configuration of points. Let $f: Y\longrightarrow \P^2$
be the blow up of $\P^2$ at points of $Z$ with the exceptional divisors $E_i$ over $P_i$ for $i=1,\ldots,9$.
Let $\Btre$ be the image of $Y$ under the morphism defined by $H^0(Y,f^*\calo_{\P^2}(4)-E_1-\ldots -E_9)$.
Then
\begin{itemize}
   \item[a)] $\Btre$ is a surface in $\P^5$ of degree $7$ with $3$ singular points.
   \item[b)] In all points of the smooth locus $(\Btre)_{sm}$, the surface is
      hypo-osculating with $$\dim\Osc_P^{(2)}((\Btre)_{sm})=4.$$
   \item[c)] There is a 1:1 correspondence between unexpected curves admitted by $Z$
      and second osculating spaces of $\Btre$ at its smooth points.
\end{itemize}
   Moreover, the divisor cut out on $\Btre$ by the second osculating space at a general point of $\Btre$
   is \emph{irreducible}.
\end{theoremA}
\begin{theoremB}
   Let $X_B'$ be the image of the morphism
   $$\eta:\P^2\ni(a:b:c)\lra (a^3: b^3: c^3: a(b^2-c^2): b(a^2-c^2): c(a^2-b^2))\in\P^5.$$
   Then
\begin{itemize}
   \item[a)] $X_B'$ is a smooth surface of degree $9$ in $\P^5$.
   \item[b)] The surface $X_B'$ is hypo-osculating with
     $$\dim\Osc_P^{(2)}(X_B')=4$$
     for all points $P\in X_B'$, except those in the $B_3$ configuration, i.e., except
     the images of points $P_1,\ldots,P_9$.
\end{itemize}
\end{theoremB}
   Part a) of  Theorem A is proved with Proposition \ref{prop: B3 as deg 7 surface},
   part b) is Proposition \ref{prop: B3 is hypo-osculating} and part c) is an easy consequence of Lemma \ref{lem: the quartic}.
   Theorem B is proved in Propositions \ref{prop:B3' eta} and \ref{prop:B3' hypo osculating}.


\section{The Shifrin sufrace}\label{sec: Shifrin}
   In this section we recall briefly an example of a surface embedded in $\P^5$
   by a \emph{complete} linear system which satisfies a Laplace equation of order $2$.
Let $\Sh\subseteq \P^5$ be the image of the mapping
$$
\varphi : \P^1\times \P^1 \ni ((s:t),(u,v))\rightarrow (stv:t^2u:s^2v:stu:s^2u:t^2v) \in\P^5.
$$
Let $u_0,\ldots,u_5$ be the coordinates in $\P^5$. Then the ideal $I(\Sh)$ defining $\Sh$ is generated by $6$
quadrics
\[
\begin{array}{lll}
u_3^2-u_1u_4, \; & u_2u_3-u_0u_4, \, & u_0u_3-u_4u_5\\
u_1u_2-u_4u_5, \; & u_0u_1-u_3u_5, \, & u_0^2-u_2u_5
\end{array}
\]
and its minimal free resolution is
$$
S/I(\Sh)\leftarrow S \leftarrow S(-2)^{\oplus 6}\leftarrow  S(-3)^{\oplus 8}\leftarrow S(-4)^{\oplus 3},
$$
where $S=\C[u_0,\ldots,u_5]$.

Note that the linear system defining $\varphi$ is of bidegree $(2,1)$. If we consider this system directly
on $\P^1\times\P^1$, then we have clearly
$$
h^0(\P^1\times\P^1,\calo (2,1))=6
$$
and for every point $P=(P_1,P_2)\in\P^1\times \P^1$ there exists a divisor with a triple point there,
namely this is
$$
D_P=2\cdot \{P_1\}\times\P^1+\P^1\times\{P_2\}.
$$
In particular this divisor is reducible and non-reduced for all points $P$. Thus, even though the existence of $D_P$
is unexpected from the point of view of counting conditions, it is easily explained by the geometry
of the surface.
\begin{property}\label{pro:reducible on Shifrin}
   All divisors cut out on the Shifrin surface by 2nd osculating spaces are reducible.
\end{property}
   This example has a remarkable property that all its second osculating spaces have the same dimension,
   which is one less than the expected one. For such surfaces
   Shifrin introduced in \cite{Shi86} the following terminology.
\begin{definition}
   We say that a surface $X\subseteq \P^5$ is \emph{perfectly} hypo-osculating if
$$
\dim\Osc_P^{(2)}(X)=4
$$
for \emph{every} point $P\in X$.
\end{definition}
   He formulated a conjecture
   \cite[Conjecture 0.6]{Shi86} to the effect that the surface
   discussed in this part is the unique smooth perfectly hypo-osculating surface in $\P^5$.
   This conjecture was proved Piene and Tai in \cite{PieTai}.
   Along these lines we propose a new Problem \ref{prob:irreducible Togliatti},
   we relax the assumption of perfectly hypo-osculating and replace it by
   irreducible hypo-osculating divisors.

\section{The Togliatti surface}
   In this section we discuss the classical example of Togliatti.
   Let $\To\subset\P^5$ be the closure of the image of the mapping
   $$\psi:\P^2\ni(x:y:z)\dashrightarrow (x^2y:x^2z:y^2x:y^2z:z^2x:z^2y)\in\P^5.$$
   Let, as before, $u_0,\ldots,u_5$ be the coordinates in $\P^5$. Then the ideal
   $I(\To)$ defining $\To$ is generated by
\[
\begin{array}{lllll}
   u_2u_4-u_0u_5,\; & u_1u_3-u_0u_5,\; & u_3u_4^2-u_1u_5^2,\; & u_0u_4^2-u_1^2u_5,\; & u_3^2u_4-u_2u_5^2\\
   u_0u_3u_4-u_1u_2u_5,\; & u_0u_3^2-u_2^2u_5,\; & u_1u_2^2-u_0^2u_3,\; & u_1^2u_2-u_0^2u_4 &\\
\end{array}
\]
   and its minimal free resolution is
   $$
   S/I(\To)\leftarrow S\leftarrow S(-2)^{\oplus 2}\oplus S(-3)^{\oplus 7}\leftarrow S(-4)^{\oplus 19}\leftarrow S(-5)^{\oplus 16}\leftarrow S(-6)^{\oplus 6}\leftarrow S(-7).
   $$
   For a more geometric description
   let $P_1=(1:0:0)$, $P_2=(0:1:0)$ and $P_3=(0:0:1)$ be the coordinate points in $\P^2$ and let $Z=\left\{P_1,P_2,P_3\right\}$.
   Identifying, as usual, global sections with homogeneous polynomials, we may take as a basis of the linear system $V=H^0(\P^2,\calo_{\P^2}(3)\otimes \cali_Z)$
   the monomials
   $$x^2y, x^2z, y^2x, y^2z, z^2x, z^2y, xyz.$$
   Let $f:Y\to\P^2$ be the blow up of $\P^2$ at $P_1,P_2,P_3$ with exceptional divisors $E_1,E_2,E_3$ respectively,
   and let $H=f^*\calo_{\P^2}(1)$. Then the following holds.
\begin{proposition}\label{prop: Togliatti}
   The linear system
   $$M=3H-E_1-E_2-E_3$$
   is very ample and determines an embedding
   $$\varphi_M:Y\to\P^6,$$
   which lifts the rational map
   $$\wtilde{\psi}:\P^2\ni(x:y:z)\dashrightarrow (x^2y:x^2z:y^2x:y^2z:z^2x:z^2y:xyz)\in\P^6,$$
   i.e., the following diagram commutes
\begin{center}
\begin{tikzcd}
 Y \arrow[d, "f"'] \arrow{rd}{\varphi_M} & \\
 \P^2 \arrow[r, dashed, "\wtilde{\psi}"] &  \P^6
\end{tikzcd}
\end{center}
   Then the Togliatti surface $\To$ is the image of $\varphi_M(Y)$ under the projection $\pi$
   from the last coordinate point $(0:0:0:0:0:0:1)\in\P^6$. We have
\begin{center}
\begin{tikzcd}
 Y \arrow[d, "f"'] \arrow{r}{\varphi_M} & \P^6 \arrow[d, dashed, "\pi"]\\
 \P^2 \arrow[r, dashed, "\psi"] \arrow[ur, dashed, "\wtilde{\psi}"] &  \P^5
\end{tikzcd}
\end{center}
\end{proposition}
\proof
   It suffices to note that $M$ is the anti-canonical bundle on $Y$, which itself is the Del Pezzo surface
   of degree $6$.
\endproof
   The hypo-osculating divisors for the Togliatti surface $\To$ have an easy geometric interpretation.
   Let $P=(a:b:c)$ be a general point in $\P^2$. We may assume, in particular that $a,b,c$ are all non-zero.
   Then the polynomial
   $$f_P(x,y,z)=(c x-a z)(c y-b z)(b x-a y)$$
   defines a degree $3$ curve (the union of $3$ lines) vanishing in the three coordinate points
   of $\P^2$ and vanishing at $P$ to order $3$. This is visualized in Figure \ref{fig: Togliatti}.
\begin{figure}[h!]
\begin{center}
\definecolor{xdxdff}{rgb}{0,0,0}
\definecolor{ududff}{rgb}{0,0,0}
\begin{tikzpicture}[line cap=round,line join=round,>=triangle 45,x=1.0cm,y=1.0cm]
\clip(1,-1) rectangle (7,4.2);
\draw [line width=2.pt,domain=3.7:4.7] plot(\x,{(--9.2-1.94*\x)/0.48});
\draw [line width=2.pt,domain=3:7] plot(\x,{(--6.36-1.06*\x)/1.52});
\draw [line width=2.pt,domain=-1:6] plot(\x,{(--2.12-1.06*\x)/-2.48});
\begin{scriptsize}
\draw [fill=ududff] (4.,3.) circle (2.5pt);
\draw[color=ududff] (4.9,3.37) node {$(0:0:1)$};
\draw [fill=xdxdff] (6.,0.) circle (2.5pt);
\draw[color=xdxdff] (6.35,0.37) node {$(0:1:0)$};
\draw [fill=xdxdff] (2.,0.) circle (2.5pt);
\draw[color=xdxdff] (2,0.37) node {$(1:0:0)$};
\draw [fill=ududff] (4.48,1.06) circle (3.5pt);
\draw[color=ududff] (4.62,1.43) node {$P$};
\end{scriptsize}
\end{tikzpicture}
\caption{$C_P$}
\label{fig: Togliatti}
\end{center}
\end{figure}

   Since
   \begin{equation}\label{eq:Togliatti}
      f_P=bc^2\cdot x^2y-b^2c\cdot x^2z-ac^2\cdot y^2x+a^2c\cdot y^2z+ab^2\cdot z^2x-a^2b\cdot z^2y,
   \end{equation}
   it corresponds to a divisor $D_P$ in $\P^5$. Since $\mult_PD_P=3$, we can identify $\Osc_P^{(2)}\To$ with $D_P$.
   Thus, similarly as in the case of the Shifrin surface, we have that all divisors $D_P$ are reducible.
\begin{property}\label{pro:reducible on Togliatti}
   All divisors cut out on the Togliatti surface by 2nd osculating spaces are reducible.
\end{property}


\section{The unexpected quartics and a $B_3$-surface in $\P^5$}
   In this section, which is the core of this note, we study a hypo-osculating surface
   in $\P^5$ whose general osculating space of order $2$ is determined by an irreducible
   divisor. We begin setting up the notation and recalling basic properties proved in \cite{CHMN}.
\subsection{The ideal of the $B_3$-surface}
Let $Z$ be the following set of points in $\P^2$:
$$
\begin{array}{lll}
   P_1=(1:0:0),& P_2=(0:1:0),& P_3=(0:0:1),\\
   P_4=(1:1:0),& P_5=(1:-1:0),& P_6=(1:0:1),\\
   P_7=(1:0:-1),& P_8=(0:1:1),& P_9=(0:1:-1).
\end{array}
$$
These points impose independent conditions on quartics in $\P^2$. The vector space $H^0(\P^2,\calo_{\P^2}(4) \otimes \cali(Z))$
is generated by the following six quartics:
\begin{equation}\label{eq:deg4}
\begin{array}{lll}
   x^2yz,& xy^2z,& xyz^2,\\
   xy(x^2-y^2),& xz(x^2-z^2),& yz(y^2-z^2)
\end{array}
\end{equation}
and their set of zeroes is exactly $Z$ (note that the saturated ideal $I(Z)$ has one generator $xyz$ of degree $3$
and three generators of degree $4$).

Let $f:Y\rightarrow \P^2$ be the blow up at points $P_1,\ldots,P_9$ with exceptional divisors $E_1,\ldots, E_9$.
Let $H=f^*\calo_{\P^2}(1)$.
\begin{proposition}\label{prop: B3 as deg 7 surface}
The linear system
$$
M=4H-E_1-\ldots -E_9
$$
is base point free and defines a morphism birational onto its image
$$
\varphi_M :Y\rightarrow \P^5
$$
which is an isomorphism away from the proper transforms of three lines
$$
\begin{array}{lll}
   L_1 : x=0 & \mbox{\rm{ through }}& P_2, P_3, P_8, P_9,\\
   L_2 : y=0 & \mbox{\rm{ through }}& P_1, P_3, P_6, P_7,\\
   L_3 : z=0 & \mbox{\rm{ through }}& P_1, P_2, P_4, P_5.
\end{array}
$$

\end{proposition}

\begin{proof}
Let
$$
\begin{array}{l}
   N_1=H-E_2-E_3-E_8-E_9,\\
   N_2=H-E_1-E_3-E_6-E_7,\\
   N_3=H-E_1-E_2-E_4-E_5
\end{array}
$$
be proper transformations of the lines $L_1, L_2, L_3$.

To begin with, we show that $M$ is nef. Note that
$$
\begin{array}{lll}
   M.N_i=0 & \mbox{\rm{ for }}& i=1,2,3.
\end{array}
$$
We have also $M.E_i=1$ for $i=1,\ldots,9$. Let $C$ be a plane curve of degree
$d$ passing through the points $P_1,\ldots,P_9$ with multiplicity $m_1,\ldots,m_9$
different from $L_1, L_2, L_3$. Let $\widetilde{C}=dH-m_1E_1-\ldots-m_9E_9$
be the proper transform of $C$. Then
$$
\begin{array}{l}
  0\leq \widetilde{C}.N_1=d-m_2-m_3-m_8-m_9,\\
  0\leq \widetilde{C}.N_2=d-m_1-m_3-m_6-m_7,\\
  0\leq \widetilde{C}.N_3=d-m_1-m_2-m_4-m_5.
\end{array}
$$
Hence
$$
3d-2(m_1+m_2+m_3)-(m_4+m_5+m_6+m_7+m_8+m_9)\geq 0.
$$
Since
$$
M.\widetilde{C}=4d-\sum^9_{i=1}m_i,
$$
we conclude that $M.\widetilde{C}>0$. So this shows not only that $M$ is nef but also
that $N_1, N_2, N_3$ are the only effective curves on $Y$ orthogonal to $M$ with respect
to the intersection form.

The surface $Y$ is an anticanonical surface. That means that $-K_Y$ is an effective divisor. Indeed,
we have
$$
-K_Y=N_1+N_2+N_3+E_1+E_2+E_3.
$$
Linear systems on anticanonical surfaces have been studied intensively by Harbourne. In
\cite[Theorem III.1.(a)]{Har97TAMS} he showed in particular that if $Y$ is an anticanonical
surface and $M$ is a nef linear system on $Y$ such that $(-K_Y.M)\geq 2$, then $M$ is non-special,
i.e. $h^1(Y,M)=0$ and $M$ is base point free. These conditions hold in our situation, since
$$
(-K_Y.M)=12-9=3.
$$
In \cite[Proposition 3.2]{Har97JA} Harbourne showed moreover that on an anticanonical surface
a nef class $M$ with $M^2>0$ and such that $(-K_Y.M)\geq 3$ defines a birational morphism
onto a surface obtained by contracting all curves perpendicular to $M$. Since $M^2=7$ in our
case, we conclude that the image of $Y$ under the morphism $\varphi_{|M|}$ is a surface $\Btre$
of degree $7$, singular at $3$ points resulting from contracting the curves $N_1, N_2$ and $N_3$.
\end{proof}
   Working with specific coordinates introduced at the beginning of this section, we can identify $\Btre$
   as the closure of the image of the rational map
\begin{equation}\label{eq:B3 psi}
   \psi:\P^2\ni(x:y:z)=(x^2yz:xy^2z:xyz^2:xy(x^2-y^2):xz(x^2-z^2):yz(y^2-z^2))\in\P^5.
\end{equation}
   In particular, the image of
$$
\begin{array}{lll}
   N_1 & \mbox{\rm{ is }}& (0:0:0:0:0:1),\\
   N_2 & \mbox{\rm{ is }}& (0:0:0:0:1:0),\\
   N_3 & \mbox{\rm{ is }}& (0:0:0:1:0:0).
\end{array}
$$
   Using the explicit description provided above, we conclude that the ideal $I(\Btre)$ of the surface $\Btre$ in $\P^5$ is generated by
\[
\begin{array}{llll}
   u_1^2-u_2^2-u_0u_5, & u_0^2-u_2^2-u_1u_4, & u_2u_3-u_1u_4+u_0u_5, & u_0u_1u_3-u_0u_2u_4+u_1u_2u_5-u_3u_4u_5
\end{array}
\]
and its minimal free resolution is
$$
S/I(\Btre)\leftarrow S \leftarrow S(-2)^{\oplus 3}\oplus S(-3)\leftarrow  S(-4)^{\oplus 6}\leftarrow S(-5)^{\oplus 3}.
$$
\subsection{Osculating spaces of the $B_3$ surface}\label{ssec: B3 osculating spaces}
   In this section we show that $\Btre$ is hypo-osculating.
\begin{proposition}\label{prop: B3 is hypo-osculating}
   The surface $\Btre$ satisfies one Laplace equation.
\end{proposition}
\proof
   Using the parametrization in \eqref{eq:B3 psi}, we determine second osculating
   spaces of $\Btre$. Let $P=(a:b:c)$ be a general point. Evaluating second order
   derivatives of $\psi$ at $P$ we get the Hesse matrix $H_P\psi$.
\begingroup
\renewcommand*{\arraystretch}{1.2}
$$\begin{array}{c|cccccc}
   & xyz^2 & yz(y^2-z^2) & xy^2z & x^2yz & xz(x^2-z^2) & xy(x^2-y^2)\\
   \hline
   \partial^2/\partial^2_x & 0 & 0 & 0 & 2bc & 6ac & 6ab \\
   \partial^2/\partial^2_y & 0 & 6bc & 2ac & 0 & 0 & -6ab \\
   \partial^2/\partial^2_z & 2ab & -6bc & 0 & 0 & -6ac & 0 \\
   \partial^2/\partial_x\partial_y & c^2 & 0 & 2bc & 2ac & 0 & 3(a^2-b^2) \\
   \partial^2/\partial_x\partial_z & 2bc & 0 & b^2 & 2ab & 3(a^2-c^2) & 0 \\
   \partial^2/\partial_y\partial_z & 2ac & 3(b^2-c^2) & 2ab & a^2 & 0 & 0
\end{array}$$
\endgroup
   It is easy to check either by hand or by a symbolic algebra computer program
   (we used Singular \cite{Singular}) that $\det H_P\psi=0$. This means that
   $\Btre$ satisfies at least $1$ Laplace equation of order $2$. Computing
   $\Lambda^i(H_P\psi)$ for $i=5,4,3$ we conclude that
   \begin{itemize}
   \item the rank of $H_P\psi$ drops by at least $2$ along the lines $abc=0$,
   \item the rank of $H_P\psi$ drops by $3$ at the $3$ coordinate points,
   \item the matrix $H_P\psi$ has always rank at least $3$.
   \end{itemize}
\endproof
\subsection{The $B_3$ surface and Lefschetz Properties}\label{ssec: B3 and Lefschetz}
   In \cite{MMO13} Mezzetti, Mir\'o-Roig and Ottaviani described a connection,
   based on Macaulay-Matlis duality (also known as apolarity),
   between projective varieties satisfying at least one Laplace equation
   and homogeneous ideals in a polynomial ring, generated by polynomials
   of equal degrees, and failing the Weak Lefschetz Property.

\begin{definition}[Weak Lefschetz Property]\label{def:WLP}
   An artinian homogeneous ideal $I$ in a polynomial ring $R=\K[x_0,\ldots,x_N]$ (i.e. an ideal such that the Krull dimension of $R/I$ is zero)
   satisfies the Weak Lefschetz Property (WLP for short), if there exists a linear form $\ell\in (R/I)_1$ such that
   for all integers $d$, the multiplication map
   \begin{equation}\label{eq:multiplication}
      \cdot\ell: \;(R/I)_{d}\ni f\to f\cdot\ell\in(R/I)_{d+1}
   \end{equation}
   has the maximal rank (i.e. the map is either injective or surjective).
\end{definition}
   Note that
   since the algebra $R/I$ is finite dimensional over the ground field $\K$ the condition
   is relevant for only finitely many values of $d$. If such an $\ell$ exists, then the
   maximal rank property holds for a general linear form (i.e. taken from a Zariski open set in the vector space $(R/I)_1$).

   There is another notion defined similarly.
\begin{definition}[Strong Lefschetz Property]
   An ideal $I$ as in Definition \ref{def:WLP} is said to satisfy the Strong Lefschetz Property (SLP for short), if there
   exists $\ell$ as above such that the multiplication map
   \begin{equation}\label{eq:multiplication2}
      \cdot\ell^k: \;(R/I)_{d}\ni f\to f\cdot\ell^k\in(R/I)_{d+k}
   \end{equation}
   has maximal rank for all $d$ and all $k\geq 1$.
\end{definition}
   An ideal $I$ is said to fail the WLP in degree $d$ if the map \eqnref{eq:multiplication}
   is not of maximal rank for all $\ell$.
\medskip

   Similarly, an ideal $I$ fails the SLP at range $k$ and degree $d$ if the map \eqnref{eq:multiplication2}
   is not of maximal rank for all $\ell$.
\medskip

   The interest in the WLP stems partly from the fact that this property imposes
   strong constraints on the Hilbert function of $I$. In fact, there is a complete
   classification of Hilbert functions of artinian algebras satisfying the WLP, \cite[Proposition 3.5]{HMNW03}.
   Many algebras are expected to have
   this property but establishing it is often very difficult. On the other hand
   ideals failing the WLP often exhibit additional peculiar properties
   of an algebraic or geometric nature or give rise to such ideals via some standard constructions,
   for example Macaulay-Matlis duality. We recall now briefly what this duality is.
   To this end let $V$ be a $\K$-vector space of dimension $N+1$ and let us identify
   the polynomial ring $R$ over $\K$ with the symmetric algebra of forms
   $$R=\oplus_{d\geq 0}\Sym^dV^*.$$
   We define the dual algebra
   $$\cald=\oplus_{d\geq 0}\Sym^dV.$$
   Let $x_0,\ldots,x_N$ and $y_0,\ldots,y_N$ be dual bases of $R$ and $\cald$, respectively.
   There is a natural pairing
   $$\Sym^jV^*\tensor \Sym^iV \ni (F,G) \lra F\cdot G \in\Sym^{i-j}V$$
   inducing on $\cald$ the structure of a graded $R$-module. We can think of this pairing as being defined
   by partial differentiation in the following way
   $$F\cdot G= F(\partial/\partial y_0,\ldots,\partial/\partial y_N) G.$$
   Usually, one does not distinguish between the coordinates in $R$ and $\cald$ as in the following Example.
\begin{example}\label{ex:MM}
   Let $F_1=x^2yz$ and $F_2=x^3y$. Let $G=(x+y)^4$. Then
   $$F_1\cdot G=\frac{\partial^4}{\partial^2_x\partial_y\partial_z}G=0,$$
   $$F_2\cdot G=\frac{\partial^4}{\partial^3_x\partial_y}(x^4+4x^3y+6x^2y^2+4xy^3+y^4)=4.$$
\end{example}
\begin{definition}[Macaulay inverse system]
   Let $I\subset R$ be a homogeneous ideal. The Macaulay inverse system $I^{-1}$ for $I$ is
   $$I^{-1}=\left\{G\in \cald:\; F\cdot G=0\mbox{ for all }F\in I\right\}.$$
\end{definition}
   For the convenience of the reader, we recall the following fundamental result
   of Mezzetti, Mir\'o-Roig and Ottaviani \cite[Theorem 3.2]{MMO13}.
\begin{theorem}\label{thm:tea}
   For an artinian ideal $I\subset R$ generated by $r\leq\binom{N+d-1}{N-1}$ forms $F_1,\ldots, F_r$ of degree $d$
   the following conditions are equivalent:
   \begin{itemize}
   \item[i)] the ideal $I$ fails the WLP in degree $d-1$;
   \item[ii)] The forms $F_1,\ldots,F_r$ become linearly dependent when restricted to a general hyperplane $H$ in $\P^N$;
   \item[iii)] The image $X$ of $\P^N$ under the morphism defined by the inverse system $I^{-1}$ satisfies at least
   one Laplace equation of order $d-1$.
   \end{itemize}
\end{theorem}
   The points in $Z$ determine in the dual projective plane a $B_3$ arrangement of lines, which is visualized
   in Figure \ref{fig:B3}. The missing ninth line
   is the line at infinity, $z=0$.
\begin{figure}[h]
\begin{center}
\begin{tikzpicture}[line cap=round,line join=round,>=triangle 45,x=1.0cm,y=1.0cm,scale=0.9]
\clip(-3,-3) rectangle (3,3);
\draw [line width=0.8pt] (0.,-8.491111111111087) -- (0.,6.508888888888871);
\draw [line width=0.8pt] (-1.,-8.491111111111087) -- (-1.,6.508888888888871);
\draw [line width=0.8pt] (1.,-8.491111111111087) -- (1.,6.508888888888871);
\draw [line width=0.8pt,domain=-9.351111111111113:10.328888888888894] plot(\x,{(-0.-0.*\x)/1.});
\draw [line width=0.8pt,domain=-9.351111111111113:10.328888888888894] plot(\x,{(-1.-0.*\x)/1.});
\draw [line width=0.8pt,domain=-9.351111111111113:10.328888888888894] plot(\x,{(--1.-0.*\x)/1.});
\draw [line width=0.8pt,domain=-9.351111111111113:10.328888888888894] plot(\x,{(-0.-1.*\x)/1.});
\draw [line width=0.8pt,domain=-9.351111111111113:10.328888888888894] plot(\x,{(-0.-1.*\x)/-1.});
\end{tikzpicture}
\caption{   $xyz(x^2-y^2)(y^2-z^2)(z^2-x^2)$}
\label{fig:B3}
\end{center}
\end{figure}
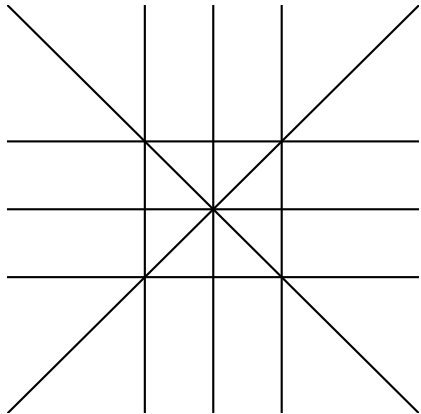
   We consider the artinian ideal $J$ generated by the fourth powers of linear
   forms defining the lines in the $B_3$ arrangement
   $$J=\langle x^4, y^4, z^4, (x-y)^4, (x+y)^4, (x-z)^4, (x+z)^4, (y-z)^4, (y+z)^4\rangle.$$
   Theorem \ref{thm:tea} does not apply directly to our situation because the
   number of generators $9$ exceeds the upper bound in its assumptions. However,
   as we shall see in a moment, condition iii) is satisfied with the order $d-2$ in place of $d-1$. Thus it is natural
   to wonder if this property is reflected in the way $J$ fails some of Lefschetz Properties.
   To this end we compute first (as in Example \ref{ex:MM}) the Macaulay-Matlis
   dual of $J$:
   $$J^{-1}=\langle x^2yz, xy^2z, xyz^2, xy(x^2-y^2), xz(x^2-z^2), yz(y^2-z^2)\rangle.$$
   Thus $J^{-1}=\langle(I(Z))_4\rangle$ by \eqref{eq:deg4}. Hence the morphism determined by $J^{-1}$
   is exactly the morphism $\psi$ defined in \eqref{eq:B3 psi}. By Proposition \ref{prop: B3 is hypo-osculating}
   the image $\Btre$ of the inverse system satisfies one Laplace equation of order $2$.

   It is convenient to write down explicitly all graded pieces of the artinian algebra $A=R/J$.
   For $i=0,\ldots,3$ we have $A_i=(R/J)_i=M^i$, where $M=\langle x,y,z\rangle$ is the
   maximal ideal and
   $$A_4=\langle xy^3, yz^3, xz^3,xyz^2, xy^2z, x^2yz\rangle.$$
   All other graded pieces of $A$ are zero.

   Thus the only interesting multiplication maps are those going to $A_4$.
   Taking as a basis of $A_2$ the monomials $x^2, y^2, z^2, xy, xz, yz$
   and a general linear form $\ell=ax+by+cz$, we can write down explicitly
   the matrix of the multiplication map
   $$\cdot \ell^2: A_2\lra A_4$$
   as follows
\begingroup
\renewcommand*{\arraystretch}{1.2}
$$\begin{array}{c|cccccc}
   & xy^3 & xz^3 & yz^3 & xyz^2 & xy^2z & x^2yz\\
   \hline
   x^2 & -2ab & -2ac & 0    & 0   & 0     & 2bc \\
   y^2 & 2ab  & 0    & -2bc & 0   & 2ac   & 0 \\
   z^2 & 0    & 2ac  & 2bc  & 2ab & 0     & 0 \\
   xy & b^2-a^2 & 0 & 0 & c^2 & 2bc & 2ac \\
   xz & 0 & c^2-a^2 & 0 & 2bc & b^2 & 2ab \\
   yz & 0 & 0 & c^2-b^2 & 2ac & 2ab & a^2
\end{array}$$
\endgroup
   This matrix has rank at most $5$ and at least $3$ for any value of $a,b,c$. The stratification
   of the degeneracy locus of $\cdot \ell^2$ is in fact exactly the same as that of $\psi$
   as described in section \ref{ssec: B3 osculating spaces}. Thus we have recovered \cite[Proposition 7.3]{DIV14}.
\begin{proposition}
   The ideal $J$ fails the SLP in range $2$ at degree $2$.
\end{proposition}
\begin{remark}
   The equivalence between the failure of the SLP in range $2$ and degree $d-2$
   and the existence of a non-trivial Laplace equation for the image
   of $\P^2$ under the mapping determined by the inverse system of $J$
   has been observed recently also by Di Gennaro and Ilardi in \cite[Corollary 25]{DigIla18}
   and in a more general setting by Di Gennaro, Ilardi and Vall\`es
   in \cite[Theorem 5.1]{DIV14}.
\end{remark}
\subsection{Unexpected quartics, the $B_3$ surface and BMSS duality}
   The surface $\Btre$ is of special interest, because in contrast to the Shifrin and the Togliatti surfaces,
   its second osculating space at a general point cuts out an irreducible divisor. Additionally,
   the BMSS duality (see \cite[Section 4]{HMNT}) associates to $\Btre$ a companion surface $\Btre'$,
   which is also a Togliatti-type surface. We describe these phenomena in this section.

   To begin with we identify the hypo-osculating divisors.
\begin{lemma}\label{lem: the quartic}
   Let $P=(a:b:c)$ be a general point in $\P^2$. Then the quartic
   \begin{align*}
      f_P(x:y:z) ={} & 3a(b^2-c^2)\cdot x^2yz +3b(c^2-a^2)\cdot xy^2z +3c(a^2-b^2)\cdot xyz^2\\
                & + a^3\cdot yz(y^2-z^2) -b^3\cdot xz(x^2-z^2) +c^3\cdot xy(x^2-y^2) \numberthis\label{eq:quartic}
   \end{align*}
   vanishes at all points of $Z$ and has a triple point at $P$.
\end{lemma}
\proof
   The existence of $f_P$ has been observed indirectly in \cite[Proposition 7.3]{DIV14}.
   This observation has been repeated explicitly in \cite{CHMN}. In fact it has been the starting
   point of research presented in this article. The explicit equation of $f_P$ was found in \cite[Section 2]{BMSS}.
\endproof
   It has been also observed in \cite{BMSS} that the equation \eqref{eq:quartic} can
   be written down as a cubic in $(a:b:c)$ variables. More precisely we have the following.
\begin{lemma}
   Let $Q=(x:y:z)$ be a general point in $\P^2$ with coordinates $(a:b:c)$.
   Then the cubic
   \begin{align*}
      g_Q(a:b:c)={} & yz(y^2-z^2)\cdot a^3 +xz(z^2-x^2)\cdot b^3 +xy(x^2-y^2)\cdot c^3   \\
                & +3x^2yz\cdot a(b^2-c^2) -3xy^2z\cdot b(a^2-c^2) +3xyz^2\cdot c(a^2-b^2)
                \numberthis\label{eq:quartic-dual}
   \end{align*}
   has a triple point at $Q$.
\end{lemma}
\proof
   It is a straightforward calculation.
\endproof
   The interplay between $f_P$ and $g_Q$ has been studied in a much more general
   setting in \cite[Section 4]{HMNT}. Here we take a slightly different point of view.
\begin{proposition}\label{prop:B3' eta}
   Let $I'=\langle a^3, b^3, c^3, a(b^2-c^2), b(a^2-c^2), c(a^2-b^2)\rangle$ and let
   $$\eta:\P^2\ni(a:b:c)\lra (a^3: b^3: c^3: a(b^2-c^2): b(a^2-c^2): c(a^2-b^2))\in\P^5$$
   be the associated rational map. Then
   \begin{itemize}
   \item $\eta$ is a morphism (i.e. the ideal $I'$ is artinian),
   \item the image of $\eta$ is a smooth surface $\Btre'$ of degree $9$,
   \item the ideal $I(\Btre')$ has the following minimal free resolution
   \end{itemize}
   $$
   S/I(\Btre')\leftarrow S \leftarrow S(-3)^{\oplus 7}\oplus S(-4)^{\oplus 2}\leftarrow  S(-4)^{\oplus 3}\oplus S(-5)^{\oplus 19}\leftarrow S(-6)^{\oplus 21} \leftarrow S(-7)^{\oplus 8}\leftarrow S(-8).
   $$
\end{proposition}
\proof
   These claims are easy to check by a symbolic algebra program.
\endproof
   We call the surface $\Btre'$ the (BMSS) \emph{companion surface} of $\Btre$.
   It enjoys several peculiar properties which are listed in the next proposition.
\begin{proposition}\label{prop:B3' hypo osculating}
   The surface $\Btre'$ is hypo-osculating. Moreover, we have
   $$\dim\Osc_Q^{(2)}\Btre'=4$$
   for all points $Q\in\P^2$ with the exception of points $Q\in Z$, where
   $\dim\Osc_{P_i}^{(2)}\Btre'=3$ for $i=1,\ldots,9$.
\end{proposition}
\proof
   The assertion follows directly from the analysis of the Hesse matrix of $\eta$
   along the same lines as in the proof of Proposition \ref{prop: B3 is hypo-osculating}.
   We omit the easy calculations.
\endproof
   Since $I'$ has $6$ generators, it is unexpected that it contains the element $g_Q$
   vanishing to order $3$ at a general point $Q$. Of course, a cubic curve with
   a triple point splits in a product of lines, so, similarly as in the case
   of the Shifrin and the Togliatti surfaces we obtain the following.
\begin{property}\label{pro:reducible on companion}
   All divisors cut out on the companion $B_3$ surface by 2nd osculating spaces are reducible.
\end{property}
\section{Final remarks}
   We conclude the note with the following two problems, which we hope to come
   back to in the near future.
\begin{problem}
   Are there companion varieties for other Togliatti-type varieties?
\end{problem}
   Note that there is a certain renaissance of interest in Togliatti systems
   and hypo-osculating varieties stemming partly from their connection to Lefschetz
   Properties and partly from the new research direction of unexpected hypersurfaces
   established in \cite{CHMN}, see also \cite{MezMir18}, \cite{DigIla18}, \cite{Mez18}.

   It follows immediately from \eqref{eq:Togliatti} that the companion surface for the Togliatti surface
   is the surface itself. In other words, the Togliatti surface is BMSS self-dual.

   Taking into account Properties \ref{pro:reducible on Shifrin}, \ref{pro:reducible on Togliatti}, \ref{pro:reducible on companion}
   and the irreducibility of hypo-osculating divisors on the $B_3$ surface, it is tempting to ask the following question,
   which seems more feasible than the Shifrin conjecture mentioned in section \ref{sec: Shifrin}.
\begin{problem}\label{prob:irreducible Togliatti}
   Does there exist a smooth Togliatti-type surface $X$ in $\P^5$ such that
   the divisor cut out on $X$ by the second osculating space at a general point is irreducible?
\end{problem}

\paragraph*{Acknowledgements.}
   I would like to thank Igor Dolgachev for suggesting that I investigate the
   surface associated to the $B_3$ root system. I thank heartily Thomas Bauer and Brian Harbourne for
   very helpful suggestions on the exposition of the manuscript. I thank also
   Jean Vall\`es for explaining to me chronological development of the theory
   and other useful comments. Finally I would like to thank Ragni Piene for
   clarifying the status of Shiffrin conjecture for me.
   The writing of this article was finished in Oberwolfach during the workshop
   ''Asymptotic invariants of homogeneous ideals'' held in October 2018.
   It is a pleasure to thank MFO for perfect working conditions.



\begin{thebibliography}{10}

\bibitem{BMSS}
T.~Bauer, G.~Malara, T.~Szemberg, and J.~Szpond.
\newblock Quartic unexpected curves and surfaces.
\newblock {\em Manuscripta Math.}, to appear,
  doi.org/10.1007/s00229-018-1091-3.

\bibitem{CHMN}
D.~Cook~II, B.~Harbourne, J.~Migliore, and U.~Nagel.
\newblock Line arrangements and configurations of points with an unexpected
  geometric property.
\newblock {\em Compos. Math.}, 154(10):2150--2194, 2018.

\bibitem{Singular}
W.~Decker, G.-M. Greuel, G.~Pfister, and H.~Sch\"onemann.
\newblock {\sc Singular} {4-1-1} --- {A} computer algebra system for polynomial
  computations.
\newblock \url{http://www.singular.uni-kl.de}, 2018.

\bibitem{DigIla18}
R.~Di~Gennaro and G.~Ilardi.
\newblock Laplace equations, {L}efschetz properties and line arrangements.
\newblock {\em J. Pure Appl. Algebra}, 222(9):2657--2666, 2018.

\bibitem{DIV14}
R.~Di~Gennaro, G.~Ilardi, and J.~Vall\`es.
\newblock Singular hypersurfaces characterizing the {L}efschetz properties.
\newblock {\em J. Lond. Math. Soc. (2)}, 89(1):194--212, 2014.

\bibitem{FGSW}
{\L }.~Farnik, F.~Galuppi, L.~Sodomaco, and W.~Trok.
\newblock On the unique unexpected quartic in {${\mathbf P}^2$}.
\newblock {\em arXiv:1804.03590}.

\bibitem{Har97TAMS}
B.~Harbourne.
\newblock Anticanonical rational surfaces.
\newblock {\em Trans. Amer. Math. Soc.}, 349(3):1191--1208, 1997.

\bibitem{Har97JA}
B.~Harbourne.
\newblock Birational morphisms of rational surfaces.
\newblock {\em J. Algebra}, 190(1):145--162, 1997.

\bibitem{HMNT}
B.~Harbourne, J.~Migliore, U.~Nagel, and Z.~Teitler.
\newblock Unexpected hypersurfaces and where to find them, arXiv:1805.10626.

\bibitem{HMNW03}
T.~Harima, J.~C. Migliore, U.~Nagel, and J.~Watanabe.
\newblock The weak and strong {L}efschetz properties for {A}rtinian
  {$K$}-algebras.
\newblock {\em J. Algebra}, 262(1):99--126, 2003.

\bibitem{Mez18}
E.~Mezzetti.
\newblock Osculating behavior of the {K}ummer surface in {$\Bbb P^5$}.
\newblock {\em Eur. J. Math.}, 4(1):372--380, 2018.

\bibitem{MezMir18}
E.~Mezzetti and R.~M. Mir\'o-Roig.
\newblock Togliatti systems and {G}alois coverings.
\newblock {\em J. Algebra}, 509:263--291, 2018.

\bibitem{MMO13}
E.~Mezzetti, R.~M. Mir\'o-Roig, and G.~Ottaviani.
\newblock Laplace equations and the weak {L}efschetz property.
\newblock {\em Canad. J. Math.}, 65(3):634--654, 2013.

\bibitem{PieTai}
R.~Piene and H.-s. Tai.
\newblock A characterization of balanced rational normal scrolls in terms of
  their osculating spaces.
\newblock In {\em Enumerative geometry ({S}itges, 1987)}, volume 1436 of {\em
  Lecture Notes in Math.}, pages 215--224. Springer, Berlin, 1990.

\bibitem{Poh62}
W.~F. Pohl.
\newblock Differential geometry of higher order.
\newblock {\em Topology}, 1:169--211, 1962.

\bibitem{Shi86}
T.~Shifrin.
\newblock The osculatory behavior of surfaces in {${\mathbf P}^5$}.
\newblock {\em Pacific J. Math.}, 123(1):227--256, 1986.

\bibitem{Tog29}
E.~Togliatti.
\newblock Alcune osservazioni sulle superfici razionali che rappresentano
  equazioni di laplace.
\newblock {\em Ann. Mat. Pura Appl.}, 25(4):325--339, 1929.

\end{thebibliography}


\bigskip
\bigskip
\small

   Justyna Szpond,
   Department of Mathematics, Pedagogical University of Cracow,
   Podchor\c a\.zych 2,
   PL-30-084 Krak\'ow, Poland.

\nopagebreak
   \textit{E-mail address:} \texttt{szpond@up.krakow.pl}


\end{document}